\documentclass[]{article}

\usepackage[utf8]{inputenc}
\usepackage{amsmath}
\usepackage{amsfonts}
\usepackage{amssymb}
\usepackage{cite}

\usepackage{graphicx}
\usepackage{caption}

\usepackage{verbatim}

\usepackage{array}

\usepackage{tabularx}

\usepackage{amsthm}

\theoremstyle{lema}

\theoremstyle{proposition}

\theoremstyle{theorem}
\newtheorem{theorem}{Theorem}[section]

\theoremstyle{theorem}
\newtheorem{remark}{Remark}[section]

\theoremstyle{corollary}

\theoremstyle{definition}
\newtheorem{definition}{Definition}[section]

\theoremstyle{example}
\newtheorem{example}{Example}[section]

\providecommand{\keywords}[1]
{
	\small	
	\textbf{\textit{Keywords---}} #1
}

\providecommand{\msc}[1]
{
	\small	
	\textbf{\textit{Mathematics Subject Classification---}} #1
}

\title{Fixed point theorem for generalized Chatterjea type mappings}
\author{Ovidiu Popescu and Cristina Maria Păcurar}
\date{}

\begin{document}
	
	\maketitle

	\begin{abstract}
		We introduce a new type of mappings in metric space which are three-point analogue of the well-known Chatterjea type mappings, and call them generalized Chatterjea type mappings. It is shown that such mappings can be discontinuous as is the case of Chatterjea type mappings and this new class includes the class of Chatterjea type mappings. The fixed point theorem for generalized Chatterjea type mappings is proven.
	\end{abstract}
	
	\keywords{Metric space, fixed point theorem, Chatterjea mappings}
	
	\msc{Primary 47H10; Secondary 47H09}

	\section{Introduction}
	
	Recently, Petrov \cite{Petrov} considered a new type of mappings in metric spaces, which can be characterized as mappings contracting perimeters of triangles, and gave a fixed point theorem for such mappings. He proved that such mappings are continuous, and constructed examples of mappings contracting perimeters of triangle which are not contraction mappings.
	
	\begin{definition}[Petrov \cite{Petrov}]
		Let $(X,d)$ be a metric space with $|X|\geq 3$. We shall say that $T:X\to X$ is a mapping contracting perimeters of triangles on $X$ if there exists $\alpha \in [0,1)$ such that the inequality 
		$$d(Tx,Ty) + d(Ty,Tz) + d(Tz,Tx) \leq \alpha [d(x,y)+d(y,z)+d(z,x)],$$
		holds for all three pairwise distinct points $x,y,x \in X$.
	\end{definition}

	\begin{theorem}[Petrov \cite{Petrov}]
		Let $(X,d)$, $|X|\geq 3$ be a complete metric space and let $T:X\to X$ be a mapping contracting perimeters of triangles on $X$. Then, $T$ has a fixed point if and only if $T$ does not possess periodic points of prime period $2$. The number of fixed points is at most $2$.
	\end{theorem}

	Moreover, Petrov and Bisht \cite{Petrov-Kannan} introduced a three point analogue of the Kannan type mappings \cite{Kannan}.
	
	\begin{definition}[Petrov and Bisht \cite{Petrov-Kannan}]
		Let $(X,d)$ be a metric space with $|X|\geq 3$. We shall say that $T:X\to X$ is a generalized Kannan type mapping on $X$ if there exists $\lambda \in [0,\frac23)$ such that the inequality 
		$$d(Tx,Ty) + d(Ty,Tz) + d(Tz,Tx) \leq \lambda [d(x,y)+d(y,z)+d(z,x)],$$
		holds for all three pairwise distinct points $x,y,x \in X$.
	\end{definition}

	It is shown that such mappings can be discontinuous as is the case of Kannan type mappings, and that the two classes of mappings are independent. Also, a fixed point theorem for generalized Kannan type mappings is proved.
	
	\begin{theorem}[Petrov and Bisht \cite{Petrov-Kannan}]
		Let $(X,d)$, $|X|\geq 3$ be a complete metric space and let the mapping $T:X\to X$ satisfy the following two conditions:
		\begin{itemize}
			\item[(i)] $T(Tx) \neq x$ for all $x \in X$ such that $Tx \neq x$;
			\item[(ii)] $T$ is a generalized Kannan type mapping on $X$.
		\end{itemize}
		Then, $T$ has a fixed point. The number of fixed points is at most $2$.
	\end{theorem}

	In \cite{Chatterjea}, Chatterjea proved the following result which gave the fixed point for discontinuous mappings:
	
	\begin{theorem}
		Let $(X,d)$ be a complete metric space and let $T :X \to X$ be a mapping such that for every $x,y \in X$ the inequality
		\begin{equation}
			d(Tx,Ty) \leq \lambda [d(x,Ty)+d(y,Tx)],
			\label{Chatterjea}
		\end{equation}
	holds, where $0 \leq \lambda < \frac12$. Then, $T$ has a unique fixed point.
	\end{theorem}

	We note that the fixed point theorems due to Banach \cite{Banach}, Kannan \cite{Kannan} and Chatterjea \cite{Chatterjea} are independent, and that the latter two characterize the completeness of the metric space (see Subrahmanyam \cite{Subrahmanyam}). The initial results proved by Chatterjea were of great interest, and many novel directions of research and generalizations emerged from the original result (see, for example, the papers \cite{Berinde,Dung,Fallahi,Faraji,Gautam,Imdad,Kanwal,Karapinar,Khan,Mukheimer,Mustafa,Popescu-1,Popescu-2,Som}). 

	In this paper, we give a three-point analogue of the Chatterjea type mapping. The ordinary Chatterjea mappings form an important subclass of these novel mappings. To emphasize the advances brought to the research field and the comprehensiveness of the newly introduced class of mappings, examples of generalized Chatterjea mappings, which are not Chatterjea mappings are constructed.

	\section{Generalized Chatterjea type mappings}

	\begin{definition}
		Let $(X,d)$ be a metric space with $|X|\geq 3$. We shall say that $T:X\to X$ is a generalized Chatterjea type mapping on $X$ if there exists $\lambda \in [0,\frac12)$ such that the inequality 
		\begin{equation}
			\begin{aligned}
			d(Tx,Ty) &+ d(Ty,Tz) + d(Tz,Tx) \leq \\ &\leq \lambda [d(x,Ty)+d(y,Tx)+d(y,Tz)+d(z,Tx)+d(z,Ty)+d(x,Tz)],
			\label{Chatterjea-general}
			\end{aligned}
		\end{equation}
		holds for all three pairwise distinct points $x,y,x \in X$.
	\end{definition}
	
	\begin{remark}
		Let $(X,d)$ be a metric space with $|X|\geq 3$, $T:X\to X$ be a generalized Chatterjea type mapping on $X$ and let $x,y,z \in X$ be pairwise distinct. Consider inequality (\ref{Chatterjea}) for the pairs $x,z$ and $y,z$:
		\begin{itemize}
		\item[]\begin{equation}
			d(Tx,Tz) \leq \lambda [d(x,Tz)+d(z,Tx)],
			\label{Chatterjea-xz}
		\end{equation}
		\item[]\begin{equation}
			d(Ty,Tz) \leq \lambda [d(y,Tz)+d(z,Ty)].
			\label{Chatterjea-yz}
		\end{equation}
		\end{itemize}
	
		Adding the left and the right parts of the inequalities (\ref{Chatterjea}), (\ref{Chatterjea-xz}) and (\ref{Chatterjea-yz}) we obtain (\ref{Chatterjea-general}). Hence, we get that every Chatterjea type mapping is a generalized Chatterjea type mapping.
	\end{remark}

	\begin{example}
		Let $X=\{x,y,z\}$, $d(x,y) = d(y,z) = d(z,x) = 1$ and let $T : X \to X$ be such that $Tx=x$, $Ty=y$ and $Tz=x$. 
		
		Since $d(Tx,Ty) = 1$ and $d(x,Ty)+d(y,Tx) = 2,$ $T$ is not a Chatterjea type mapping. 
		
		However,
		$$M(x,y,z) = d(Tx,Ty) + d(Ty,Tz) + d(Tz,Tx) = 2$$
		and 
		$$N(x,y,z) = d(x,Ty)+d(y,Tx)+d(y,Tz)+d(z,Tx)+d(z,Ty)+d(x,Tz) = 5,$$
		so we have 
		$$  M(x,y,z) \leq \frac25 N(x,y,z).$$
		Therefore, $T$ is a generalized Chatterjea type mapping ($\lambda = \frac25$). We note that in this case, $T$ has two fixed points.
	\end{example}

	The next theorem is the main result of the current paper which ensures that every generalized Chatterjea type mapping that does not have periodic points of period two, has fixed points, and the number of fixed points is at most two.

	\begin{theorem}
		Let $(X,d)$ be a complete metric space with  $|X|\geq 3$ and let the mapping $T:X\to X$ satisfy the following two conditions:
		\begin{itemize}
			\item[(i)] $T(Tx) \neq x$ for all $x \in X$ such that $Tx \neq x$;
			\item[(ii)] $T$ is a generalized Chatterjea type mapping on $X$.
		\end{itemize}
		Then, $T$ has a fixed point. The number of fixed points is at most two.
	\end{theorem}
	\begin{proof}
		Let $x_0 \in X$, $x_1=Tx_0$, $x_2=Tx_1$, $\dots$, $x_{n+1}=Tx_n$. Suppose that $x_n$ is not a fixed point of the mapping $T$ for every $n = 0,1,\dots$. Then, we have $x_{n} = Tx_{n-1}\neq x_{n-1}$ and $x_{n+1} = T(Tx_{n-1}) \neq x_{n-1}$ for every $n=1,2,\dots$. Hence, $x_{n-1}$, $x_n$ and $x_{n+1}$ are pairwise distinct. Taking in (\ref{Chatterjea-general}) $x=x_{n-1}$, $y=x_n$, $z=x_{n+1}$ we obtain
		\begin{equation*}
			\begin{aligned}
				d(Tx_{n-1},Tx_{n}) &+ d(Tx_{n},Tx_{n+1}) + d(Tx_{n+1},Tx_{n-1}) \leq \\ &\leq \lambda [d(x_{n-1},Tx_{n})+d(x_{n},Tx_{n-1})+d(x_{n},Tx_{n+1})+\\ &+d(x_{n+1},Tx_{n-1})+d(x_{n+1},Tx_{n})+d(x_{n-1},Tx_{n+1})],
			\end{aligned}
		\end{equation*}
		by where we get 
		\begin{equation*}
			\begin{aligned}
				d(x_{n},x_{n+1}) &+ d(x_{n+1},x_{n+2}) + d(x_{n},x_{n+2}) \leq \\ &\leq \lambda [d(x_{n-1},x_{n+1})+d(x_{n},x_{n+2})+d(x_{n},x_{n+1})+d(x_{n-1},x_{n+2})].
			\end{aligned}
		\end{equation*}
		
		Hence, we have
		\begin{equation*}
			\begin{aligned}
				(1-\lambda)[d(x_{n},x_{n+1}) &+ d(x_{n+1},x_{n+2}) + d(x_{n},x_{n+2})] \leq \\ &\leq \lambda [d(x_{n-1},x_{n+1})+d(x_{n-1},x_{n+2})].
			\end{aligned}
		\end{equation*}
	
		Using the triangle inequality
		$$d(x_{n-1},x_{n+2}) \leq d(x_{n-1},x_{n}) + d(x_{n},x_{n+1}) + d(x_{n+1},x_{n+2}),$$
		we get
		\begin{equation*}
			\begin{aligned}
				(1-\lambda)[d(x_{n},x_{n+1}) &+ d(x_{n+1},x_{n+2}) + d(x_{n},x_{n+2})] \leq \\ &\leq \lambda [d(x_{n-1},x_{n+1})+d(x_{n-1},x_{n})+d(x_n,x_{n+1})],
			\end{aligned}
		\end{equation*}
		by where
		\begin{equation*}
			\begin{aligned} d(x_{n},x_{n+1})&+d(x_{n},x_{n+2})+d(x_{n+1},x_{n+2}) \leq \\ &\leq \dfrac{\lambda}{1-\lambda}[d(x_{n-1},x_{n+1})+d(x_{n-1},x_{n})+d(x_n,x_{n+1})].
			\end{aligned}
		\end{equation*}
		
		Further, set for every $n=0,1,\dots$
		$$d_n= d(x_{n},x_{n+1})+d(x_{n},x_{n+2})+d(x_{n+1},x_{n+2}).$$
		
		Then, we have $d_n \leq \alpha d_{n-1}$ for every $n = 1,2,\dots$, where $\alpha = \dfrac{\lambda}{1-\lambda} \in [0,1)$. Hence, we get 
		$$d_n \leq \alpha d_{n-1} \leq \alpha^2 d_{n-2} \leq \dots \leq \alpha^n d_0.$$
		
		Like in the proof of Theorem 2.4 from \cite{Petrov} we obtain that $\{x_n\}$ is a Cauchy sequence. By completeness of $(X,d)$ we get that $\{x_n\}$ has a limit $x^* \in X$. Let us prove that $Tx^* = x^*$.
		
		Since $x_{n}$, $x_{n+1}$ and $x_{n+2}$ are pairwise distinct for every $n = 0,1,\dots$, there exists a subsequence $\{x_{n(k)}\}_{k \geq 0}$ such that $x_{n(k)}$, $x_{n(k) +1}$ and $x^*$ are pairwise distinct for every $k =0,1,\dots$. Taking in (\ref{Chatterjea-general}) $x=x_{n(k)}$, $y=x_{n(k)+1}$ and $z=x^*$, we obtain 
		\begin{equation*}
			\begin{aligned}
				d(Tx_{n(k)},Tx_{n(k)+1}) &+ d(Tx_{n(k)+1},Tx^*) + d(Tx^*,Tx_{n(k)}) \leq \\ &\leq \lambda [d(x_{n(k)},Tx_{n(k)+1})+d(x_{n(k)+1},Tx_{n(k)})+d(x_{n(k)+1},Tx^*)+\\ &+d(x^*,Tx_{n(k)})+d(x^*,Tx_{n(k)+1})+d(x_{n(k)},Tx^*)].
			\end{aligned}
		\end{equation*}
		Hence,
		\begin{equation*}
			\begin{aligned}
				d(x_{n(k)+1},x_{n(k)+2}) &+ d(x_{n(k)+1},Tx^*) + d(x_{n(k)+2},Tx^*) \leq \\ &\leq \lambda [d(x_{n(k)},x_{n(k)+2})+d(x_{n(k)+1},Tx^*)+d(x^*,x_{n(k)+1})+\\&+d(x^*,x_{n(k)+2})+d(x_{n(k)},Tx^*)].
			\end{aligned}
		\end{equation*}
		
		Taking the limit as $k \to \infty$ we get 
		$$2d(x^*,Tx^*) \leq 2\lambda d(x^*,Tx^*),$$
		by where $d(x^*,Tx^*)=0$, so $x^*$ is a fixed point of $T$.

		Now, suppose that there exist at least three distinct fixed points $x,y$ and $z$. Then $Tx=x$, $Ty=y$ and $Tz=z$. By (\ref{Chatterjea-general}) we have 
		$$d(x,y) + d(y,z) + d(z,x) \leq 2 \lambda[d(x,y)+d(y,z)+d(z,x)],$$
		which is a contradiction.
	\end{proof}

	\begin{remark}
		Suppose that under the assumption of Theorem 2.1, the mapping $T$ has a fixed point $x^*$ which is a limit of some iteration sequence $x_0$, $x_1=Tx_0$, $x_2=Tx_1$ $\dots$ such that $x_n \neq x^*$ for all $n = 1,2,\dots$. Then, $x^*$ is a unique fixed point. 
		
		Indeed, suppose that $T$ has another fixed point $x^{**} \neq x^*$. Obviously, there exists $N \geq 1$ such that $x_n \neq x^*$ for all $n\geq N$. Taking in (\ref{Chatterjea-general}) $x=x_n$, $y=x^*$ and $z=x^{**}$ we obtain
		\begin{equation*}
			\begin{aligned}
				d(Tx_n,Tx^*) &+ d(Tx^*,Tx^{**}) + d(Tx^{**},Tx_n) \leq \\ &\leq \lambda [d(x_n,Tx^*)+d(x^*,Tx_n)+d(x^*,Tx^{**})+\\&+d(x^{**},Tx_n)+d(x^{**},Tx^*)+d(x_n,Tx^{**})],
			\end{aligned}
		\end{equation*}
		by where 
		\begin{equation*}
			\begin{aligned}
				d(x_{n+1},x^*) &+ d(x^*,x^{**}) + d(x_{n+1},x^{**}) \leq \\ &\leq \lambda [d(x_{n},x^*)+d(x_{n+1},x^*)+ d(x_{n},x^{**})+ d(x_{n+1},x^{**})+2d(x^*,x^{**})].
			\end{aligned}
		\end{equation*}
		
		Taking the limit as $n \to \infty$, we get
		$$2d(x^*,x^{**}) \leq 4\lambda d(x^*,x^{**}),$$
		so $d(x^*,x^{**}) = 0$, which is a contradiction.
	\end{remark}

	We present two examples of generalized Chatterjea type mappings, which are not a Chatterjea type mappings, neither mappings contracting perimeters of triangles, nor generalized Kannan mappings, nor Kannan mappings. These examples highlight the contribution to the theory of fixed point theorems brought by Theorem 2.1.

	\begin{example}
		Let $X= \{A,B,C,D,E\}$ and, as in Figure \ref{Ex}, let  $$d(A,B) = d(A,C) = d(B,C)=d(B,D)=d(C,E)= d(D,E)= d(D,F)=d(E,F)=1,$$ $$d(A,D)= d(A,E)=d(B,E)= d(C,D) = d(B,F) = d(C,F) = 2,$$
		and $d(A,F) = 3$. 
		
		\begin{figure}[h!]
			\centering
			\includegraphics[width=14cm]{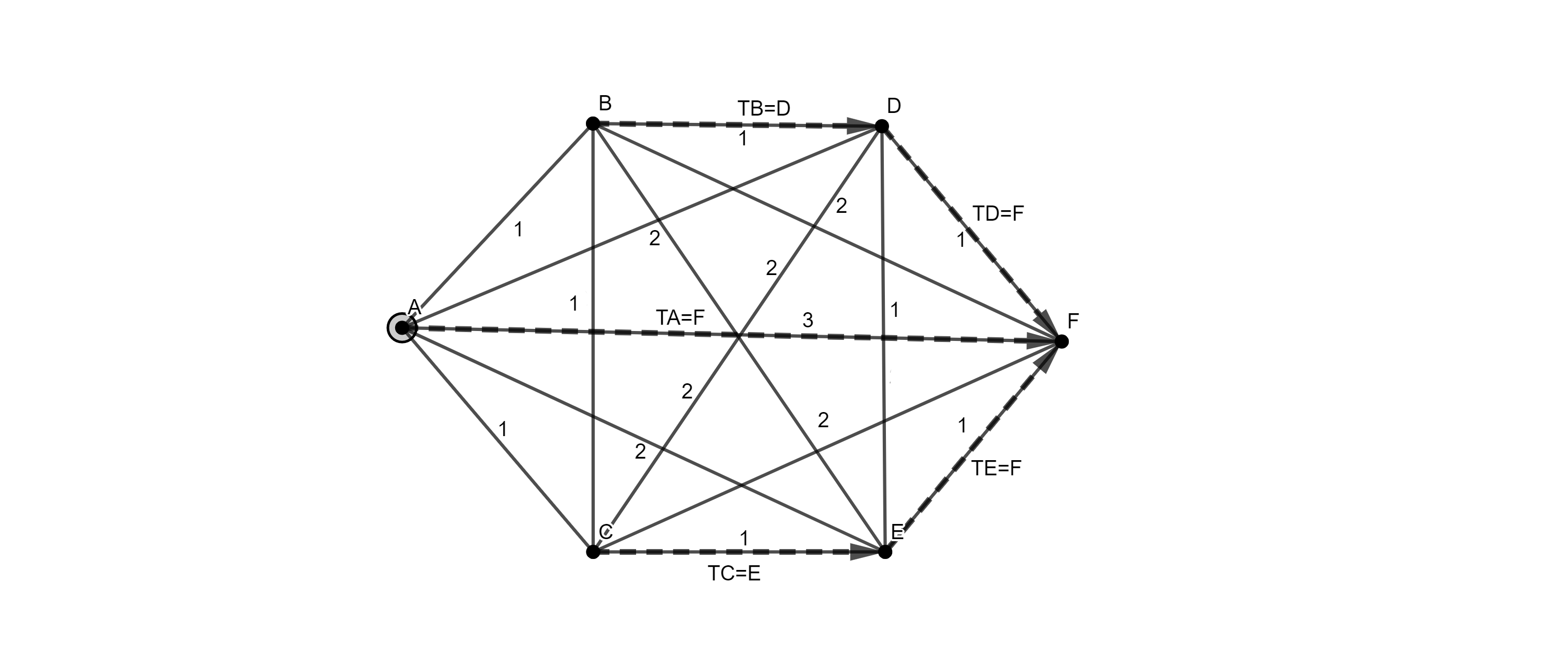}
			\caption{A generalized Chatterjea mapping}
			\label{Ex}
		\end{figure}
		
		Let $T: X \to X$ be such that 
		$TA=TD=TE=TF=F$, $TB=D$ and $TC=E$.

		We have
		\begin{equation*}
				M(A,B,C) = M(B,C,D) = M(B,C,E) = M(B,C,F) = 3,
		\end{equation*}
		\begin{equation*}
			\begin{aligned}
			& \; \; M(A,B,D) =  M(A,B,F)= M(A,B,E) = M(A,C,D)  =  \\ & = M(A,C,E)= M(A,C,F) = M(B,D,E) = M(B,D,F) =  \\ & M(B,E,F)= M(C,D,E) = M(C,D,F) = M(C,E,F) = 2,
			\end{aligned}
		\end{equation*}
		\smallskip
		\begin{equation*}
			M(A,D,E) = M(A,D,F) = M(A,E,F) = M(D,E,F) = 0,
		\end{equation*}
		and
		\begin{equation*}
			N(A,B,C) = 12,\quad 
		\end{equation*}
		\begin{equation*}
			 N(A,B,E) = N(A,C,D)= 11,
		\end{equation*}
		\begin{equation*}  
			N(A,B,D) = N(A,B,F) = N(A,C,E) = N(A,C,F) = N(B,C,F) = 10,
		\end{equation*}
		\begin{equation*}
			N(B,C,D) = N(B,C,E) =  9,
		\end{equation*}
		\begin{equation*} 
			N(B,D,E) = N(B,E,F) = N(C,D,E) = N(C,D,F) = 7, 
		\end{equation*}
		\begin{equation*}
			 N(B,D,F) = N(C,E,F) =6.
		\end{equation*}
	
		We note that $$M(x,y,z) \leq \dfrac13 N(x,y,z)$$ for all three distinct points $x,y,z \in X$, so $T$ is a generalized Chatterjea type mapping.
		
		Since 
		\begin{equation*}
			d(TB,TD) = d(D,F) = 1,
		\end{equation*}
		and 
		\begin{equation*}
			d(B,TD) + d(D,TB) = d(B,F) + d(D,D) = 2,
		\end{equation*}
	 	$T$ is not a Chatterjea type mapping.
	 	
		We note that $T$ has a unique fixed point.
		 Moreover, 
		\begin{equation*}
			d(TA,TB) +d(TA,TC) + d(TB,TC) =3
		\end{equation*} 
		and
		\begin{equation*}
			d(A,B) +d(A,C) + d(B,C) =3,
		\end{equation*} 
		so $T$ is not a mapping contracting perimeters of triangles.
		
		Since
		\begin{equation*}
			d(TB,TC) +d(TB,TD) + d(TC,TD) =3
		\end{equation*} 
		and
		\begin{equation*}
			d(B,TB) +d(C,TC) + d(D,TD) =3,
		\end{equation*} 
		$T$ is not a generalized Kannan type mapping.
		
		Also, $d(TB,TC) = 1$ and $d(B,TB)+d(C,TC)=2$, so $T$ is not a Kannan mapping.
	\end{example}
	
	\begin{example}
		Let $X=\mathbb{R}$, $d(x,y) = |x-y|$ and $T:X \to X$ defined as
		$$	Tx= \begin{cases}
		0, \quad x<2\\
		1, \quad x\geq 2.
		\end{cases}
		$$ 
		
		If $x,y,z$ are pairwise distinct and $x<2$, $y<2$, $z<2$, then $M(x,y,z) = 0$. 
		
		If $x,y,z$ are pairwise distinct and $x\geq 2$, $y\geq 2$, $z\geq 2$, then $M(x,y,z) =0$.
		
		For $x<y<2\leq z$, we have $M(x,y,z) = 2$ and $N(x,y,z) = |x| +|y|+|x-1|+z+|y-1|+z \geq 2z+2 \geq 6$, so 
		$$M(x,y,z) \leq \dfrac13 N(x,y,z).$$
		
		For $x<2\leq y < z$, we have $M(x,y,z) = 2$ and $N(x,y,z) = |x-1|+y+|x-1|+z+y-1+z-1 \geq 2y+2z-2 > 6$, so 
		$$M(x,y,z) \leq \dfrac13 N(x,y,z).$$
		
		Hence, $T$ is a generalized Chatterjea type mapping.
		
		We note that $T$ has a unique fixed point.
		
		Since
		\begin{equation*}
			M(1.9,2,2.1) = 2
		\end{equation*}	
		and
		\begin{equation*}
			d(1.9,2)+d(1.9,2.1)+d(2,2.1) =0.4,
		\end{equation*}
		we get that $T$ is not a mapping contracting perimeters of triangles.
		
		Since 
		\begin{equation*}
			M(0,1,2) = 2
		\end{equation*}	
		and
		\begin{equation*}
			d(0,T0)+d(1,T1)+d(2,T2) =2,
		\end{equation*}
		we obtain that $T$ is not a generalized Kannan  mapping.
		
		Also, since $d(T1,T2) = 1$ and $d(1,T1)+ d(2,T2) = 2$, $T$ is not a Kannan mapping. 
		
		Moreover, 
		$$d(1,T2) + d(2,T1) = 2,$$ so $T$ is not a Chatterjea mapping.
		
		We note that $T$ is a discontinuous mapping.
	\end{example}

	\medskip
	\vspace{1.2ex}
	

\begin{thebibliography}{9}
		
		\bibitem{Banach}
		S. Banach, Sur les opérations dans les ensembles abstraits et leur application aux équations intégrales, Fundamenta Mathematicae 3 (1922), 133-181.
		
		\bibitem{Berinde}
		V. Berinde, M. Păcurar, Approximating fixed points of enriched Chatterjea contractions by Krasnoselskij iterative algorithm in Banach spaces. J. Fixed Point Theory Appl. 23 (2021), 66 .
		
		
		\bibitem{Chatterjea}
		S. K. Chatterjea, Fixed-point theorems, C.R. Acad. Bulgare Sci. 25 (1972), 727–730.
		
		\bibitem{Dung}
		N. Van Dung, A. Petruşel, On iterated function systems consisting of Kannan maps, Reich maps, Chatterjea type maps, and related results. J. Fixed Point Theory Appl. 19 (2017), 2271–2285. 
		
		\bibitem{Fallahi}
		K. Fallahi, A. Aghanians, Fixed points for Chatterjea contractions on a metric space with a graph. International Journal of Nonlinear Analysis and Applications 7(2) (2016), 49-58. 
		
		\bibitem{Faraji}
		H. Faraji, K. Nourouzi, A generalization of Kannan and Chatterjea fixed point theorems on complete-metric spaces. Sahand Communications in Mathematical Analysis, 06(1) (2017), 77-86. 
		
		\bibitem{Gautam}
		P. Gautam, V. N. Mishra, R. Ali, S. Verma, Interpolative Chatterjea and cyclic Chatterjea contraction on quasi-partial b-metric space[J]. AIMS Mathematics, 2021, 6(2) (2021), 1727-1742. 
		
		\bibitem{Imdad}
		M. Imdad, A. Erduran, Suzuki-Type Generalization of Chatterjea Contraction Mappings on	Complete Partial Metric Spaces, Journal of Operators (2013), 923843.
		
		\bibitem{Kannan}
		R. Kannan, Some results on fixed point — II, Amer. Math. Monthly, 76 (1969), pp. 405-408.
		
		\bibitem{Kanwal}
		S. Kanwal, H. Isik, S. Waheed, Generalized fixed points for fuzzy and nonfuzzy mappings in strong b-metric spaces. J Inequal Appl 2024, 22 (2024). 
		
		\bibitem{Karapinar}
		E. Karapınar, H.K. Nashine, Fixed Point Theorem for Cyclic Chatterjea Type Contractions, Journal of Applied Mathematics (2012), 165698.
		
		\bibitem{Khan}
		S. H. Khan, M. Abbas, M., T. Nazir, Fixed point results for generalized Chatterjea type contractive conditions in partially ordered G-metric spaces. The Scientific World Journal  (2014), 341751. 
		
		\bibitem{Mukheimer}
		A. Mukheimer, Some Common Fixed Point Theorems in Complex Valued Metric Spaces, The Scientific World Journal (2014), 587825
		
		\bibitem{Mustafa}
		Z. Mustafa, J.R. Roshan, V. Parvaneh et al., Fixed point theorems for weakly T-Chatterjea and weakly T-Kannan contractions in b-metric spaces. J Inequal Appl 46 (2014). 
		
		\bibitem{Petrov}
		E. Petrov, Fixed point theorem for mappings contracting perimeters of triangles.
		J. Fixed Point Theory Appl. 25, 1–11 . (paper no. 74)

		
		\bibitem{Petrov-Kannan}
		E. Petrov, R. K. Bisht, Fixed point theorem for generalized Kannan type mappings, (2023), arXiv:2308.05419.
		
		\bibitem{Popescu-1}
		O. Popescu, Fixed point theorem in metric spaces, Bulletin of the Transilvania University of Braşov III 1(50)  (2008), 479–482.
		
		\bibitem{Popescu-2}
		O. Popescu, G. Stan, Some Remarks on Reich and Chatterjea Type Nonexpansive Mappings. Mathematics 8 (2020), 1270. 
		
		\bibitem{Som}
		S. Som, A. Petruşel, H. Garai et al., Some characterizations of Reich and Chatterjea type nonexpansive mappings. J. Fixed Point Theory Appl. 21 (2019), 94. 
		
		\bibitem{Subrahmanyam}
		P. V. Subrahmanyam, Completeness and fixed points, Monatsh. Math. 80 (1975), 325–330.
	
		
	\end{thebibliography}
\end{document}